\documentclass{amsart}
\title{Combinatorial identities involving the M\"obius function}
\usepackage{amssymb,amsmath,amsthm,epsfig,graphics,latexsym}

  \theoremstyle{plain}
  \newtheorem{definition}             {Definition}
  \newtheorem{lemma}      [definition]{Lemma}
  
  \newtheorem{theorem}    [definition]{Theorem}
  \newtheorem{corollary}  [definition]{Corollary}

  \theoremstyle{remark}

\begin{document}

  \author{M. El Bachraoui \and M. Salim}
  \address{Dept. Math. Sci,
 United Arab Emirates University, PO Box 17551, Al-Ain, UAE}
 \email{melbachraoui@uaeu.ac.ae}
 \email{msalim@uaeu.ac.ae}
 \thanks{Corresponding author: M. El Bachraoui, E-mail: melbachraoui@uaeu.ac.ae}
  %
  \begin{abstract}
  In this paper we derive some identities and inequalities on the M\"obius mu function. Our main
  tool is phi functions for intervals of positive integers and their unions.
  \end{abstract}
  \date{\textit{\today}}
  \maketitle
  %
%
 \section{Introduction} \label{sec:introduction}
 The \emph{M\"obius mu} function $\mu$ is an important arithmetic function
 in number theory and combinatorics which appears in various
 identities. We mention
 the following identities which are well-known and can be found
 in books on elementary number theory and arithmetic functions. Let $n$ be a positive integer. Then
 \[
 \sum_{d|n} \mu(d) =
 \begin{cases}
 1,\ \text{if $n=1$} \\
 0,\ \text{if $n>1$.}
 \end{cases}
 \]
 If $\tau(n)$ is the number of divisors of $n$, then
 \[
 \sum_{d|n} \mu(d) \tau(n/d) = 1.
 \]
 If $n= \prod_{i=1}^{r} p_i ^{k_i}$ is the prime decomposition of $n$, then
 \[
 \sum_{d|n} \mu(d) \lambda (d) = 2^r,
 \]
 where $\lambda$ denotes the \emph{Liouville lambda} function defined as follows:\\
 If $m=\prod_{i=1}^{s} p_i ^{l_i}$ is the prime decomposition of $m$, then
 \[ \lambda (m) = (-1)^{\sum_{i=1}^{s} l_i}. \]
 In this note we give some other identities on the M\"obius mu function. Our proofs are
 combinatorial with phi functions as a main tool.
 For a survey on combinatorial identities we refer to \cite{Hall, Riordan} and their references.
 We now list some examples of identities which we intend to prove.
 Let $m$ and $n$ be positive integers such that $n>1$. Then

 \noindent
 1)
 \begin{align*}
 \sum_{d|n}\mu(d) 2^{\lfloor m/d \rfloor - \lfloor (m-1)/d \rfloor} &=
 \sum_{d|n}\mu(d) \left(\lfloor m/d \rfloor - \lfloor (m-1)/d \rfloor \right) \\
 &=
 \begin{cases}
 0,\ \text{if\ } \gcd(m,n) >1 \\
 1,\ \text{if\ } \gcd(m,n)=1.
 \end{cases}
 \end{align*}

 \noindent
 2)
 \begin{align*}
 \sum_{d|n}\mu(d) 2^{\lfloor (m+1)/d \rfloor - \lfloor (m-1)/d \rfloor} &=
 \sum_{d|n}\mu(d) \binom{\lfloor (m+1)/d \rfloor - \lfloor (m-1)/d \rfloor +1}{2} \\
 &= 1 + \sum_{d|n}\mu(d) \lfloor (m+1)/d \rfloor - \lfloor (m-1)/d \rfloor \\
 &=
 \begin{cases}
  1, \ \text{if $(m,n)>1$ and $(m+1,n)>1$} \\
  2,\ \text{if $(m,n)=1$ and $(m+1,n)>1$
           or $(m,n)>1$ and $(m+1,n)=1$} \\
  3,\ \text{if $(m,n)=(m+1,n)=1$}.
  \end{cases}
 \end{align*}
 For the sake of completeness we include the following result which is a natural
 extension of \cite[Theorem 2 (a)]{ElBachraoui1} on M\"obius inversion for arithmetical functions in several variables.
 For simplicity we let
 \[
 (\overline{m}_a,\overline{n}_b)= (m_1,m_2,\ldots,m_a,n_1,n_2,\ldots,n_{b})
 \]
 and
 \[
 \left( \frac{\overline{m}_a}{d}, \left[\frac{\overline{n}_b}{d}\right]\right) =
 \left( \frac{m_1}{d},\frac{m_2}{d}, \ldots, \frac{m_a}{d},
 \left[\frac{n_1}{d}\right],\left[\frac{n_2}{d}\right],\ldots,\left[\frac{n_{b}}{d}\right]\right).
 \]
 \begin{theorem} \label{thm:inversion}
 If $F$ and $G$ are arithmetical of $a+b$ variables,
 then
 \[
 G(\overline{m}_a,\overline{n}_b)=\sum_{d|(m_1,m_2,\ldots,m_a)}F
 \left(\frac{\overline{m}_a}{d},
\left[\frac{\overline{n}_b}{d}\right]\right)\]
if and only if
\[
 F(\overline{m}_a,\overline{n}_b)=\sum_{d|(m_1,m_2,\ldots,m_a)}\mu(d)G\left(\frac{\overline{m}_a}{d},
 \left[\frac{\overline{n}_b}{d}\right]\right).
 \]
 \end{theorem}
\section{Phi functions} \label{sec:phifunction}
 Throughout let $k$, $l$, $m$, $l_1$, $l_2$, $m_1$, $m_2$ and $n$ be positive integers such that $l\leq m$,
 $l_1\leq m_1$ and $l_2\leq m_2$, let
 $[l,m]=\{l,l+1,\ldots,m\}$, and let $A$ be a nonempty finite set of positive integers.
 The set $A$ is called \emph{relatively prime to $n$}
  if $\gcd(A\cup \{n\}) = \gcd(A,n) = 1$.
  \begin{definition}
  Let
  \[ \Phi(A,n) = \# \{X\subseteq A:\ X\not=
  \emptyset\ \text{and\ } \gcd(X,n) = 1 \}
  \]
  and
  \[ \Phi_k (A,n) = \# \{X\subseteq A:\ \# X=
  k \ \text{and\ } \gcd(X,n) = 1 \}.
  \]
 \end{definition}
 Nathanson, among other things, introduced $\Phi(n)$ and $\Phi_k(n)$ (in our terminology $\Phi([1,n],n)$
 and $\Phi_k([1,n],n)$ respectively) along with their formulas in \cite{Nathanson}.
 Formulas for $\Phi([m,n],n)$ and $\Phi_k([m,n],n)$ can be found in \cite{ElBachraoui1, Nathanson_Orosz}
  and formulas for $\Phi([1,m],n)$ and $\Phi_k([1,m],n)$ are obtained in
  \cite{ElBachraoui2}.
 Ayad and Kihel in \cite{Ayad-Kihel1, Ayad-Kihel2} considered extensions to sets in arithmetic progression and obtained
 formulas for $\Phi([l,m],n)$ and $\Phi_k ([l,m],n)$ as consequences.
 Recently the following formulas for $\Phi([1,m_1]\cup [l_2,m_2])$ and $\Phi_k([1,m_1]\cup [l_2,m_2])$ have been found in
 \cite{ElBachraoui3}.
 \begin{theorem}\label{thm:main1}
 We have
  \[
 \begin{split} (a)\quad
  \Phi([1,m_1]\cup [l_2,m_2],n) &= \sum_{d|n}\mu(d)
   2^{\lfloor \frac{m_1}{d} \rfloor +\lfloor \frac{m_2}{d} \rfloor - \lfloor\frac{l_2 -1}{d} \rfloor}, \\
  (b)\quad
  \Phi_k ([1,m_1]\cup [l_2,m_2],n) &=
   \sum_{d|n} \mu(d)
   \ \binom{\lfloor \frac{m_1}{d} \rfloor +\lfloor \frac{m_2}{d} \rfloor - \lfloor\frac{l_2 -1}{d} \rfloor}{k}.
   \end{split}
  \]
 \end{theorem}
 \section{Phi functions for $[l_1,m_1]\cup [l_2,m_2]$}
 We need the following two lemmas.
 \begin{lemma} \label{lemma1}
 Let
 \[
 \Psi(l_1,m_1,l_2,m_2, n)= \# \{X \subseteq [l_1,m_1]\cup [l_2,m_2]:\ l_1, l_2\in X\
 \text{and\ } \gcd(X,n)=1 \},
 \]
 \[
 \Psi_k(l_1,m_1,l_2,m_2,n)=\# \{X \subseteq [l_1,m_1]\cup [l_2,m_2]:\ l_1,l_2\in X,\ |X| = k,\
 \text{and\ } \gcd(X,n)=1 \}.
 \]
 Then
 \[
 \text{(a)\quad } \Psi(l_1,m_1,l_2,m_2, n) = \sum_{d|(l_1,l_2,n)}\mu(d)
  2^{\lfloor \frac{m_1}{d} \rfloor + \lfloor \frac{m_2}{d} \rfloor- \frac{l_1+l_2}{d}}, \]
\[
 \text{(b)\quad } \Psi_k (m_1,l_2,m_2, n) = \sum_{d|(l_1,l_2,n)}\mu(d)
 \binom{\lfloor \frac{m_1}{d} \rfloor + \lfloor \frac{m_2}{d} \rfloor- \frac{l_1+l_2}{d}}{k-2}.
 \]
 \end{lemma}
 \begin{proof}
 (a) Assume first that $m_2\leq n$. Let $\mathcal{P}(l_1,m_1,l_2,m_2)$ denote the set of subsets of
 $[l_1,m_1]\cup[l_2,m_2]$
 containing $l_1$ and $l_2$ and let $\mathcal{P}(l_1,m_1,l_2,m_2,d)$ be the set of subsets $X$ of
 $[l_1,m_1]\cup[l_2,m_2]$ such that
 $l_1,l_2\in X$ and $\gcd(X,n) = d$. It is clear that the set $\mathcal{P}(l_1,m_1,l_2,m_2)$
 of cardinality $2^{m_1+m_2-l_1-l_2}$  can be
 partitioned using the equivalence relation of having the same $\gcd$ (dividing $l_1$, $l_2$ and $n$).
 Moreover, the mapping
 $A \mapsto \frac{1}{d} X$
 is a one-to-one correspondence between
 $\mathcal{P}(l_1,m_1,l_2,m_2,d)$ and the
 set of subsets $Y$ of
 $[l_1/d, \lfloor m_1/d \rfloor ]\cup [l_2/d,\lfloor m_2/d \rfloor]$
 such that $l_1/d, l_2/d \in Y$ and $\gcd(Y,n/d)= 1$. Then
 \[
  \# \mathcal{P}(l_1,m_1,l_2,m_2,d) =
  \Psi(l_1/d, \lfloor m_1/d \rfloor,l_2 /d,\lfloor m_2/d \rfloor,n/d).
 \]
 Thus
 \[
 2^{m_1+m_2-l_1-l_2} = \sum_{d|(l_1,l_2,n)} \# \mathcal{P}(l_1,m_1,l_2,m_2,d)=
 \sum_{d|(l_1,l_2,n)} \Psi (l_1/d,\lfloor m_1/d \rfloor,l_2 /d,\lfloor m_2/d \rfloor,n/d),
 \]
 which by Theorem \ref{thm:inversion} is equivalent to
 \[
 \Psi(l_1,m_1,l_2,m_2,n) = \sum_{d|(l_1,l_2,n)}\mu(d)
 2^{\lfloor m_1/d \rfloor + \lfloor m_2/d \rfloor - (l_1+l_2)/d}.
 \]
 Assume now that $m_2 >n$ and let $a$ be a positive integer such that $m_2 \leq n^a$.
 As $\gcd(X,n)=1$ if and only if $\gcd(X,n^a)=1$ and $\mu(d) =0$ whenever $d$ has a nontrivial
 square factor, we have
 \[
 \begin{split}
 \Psi(l_1,m_1,l_2,m_2,n) &= \Psi(l_1,m_1,l_2,m_2,n^a) \\
 &=
 \sum_{d|(l_1,l_2,n^a)}\mu(d)
 2^{\lfloor m_1/d \rfloor + \lfloor m_2/d \rfloor - (l_1+l_2)/d} \\
 &=
 \sum_{d|(l_1,l_2,n)}\mu(d)
 2^{\lfloor m_1/d \rfloor + \lfloor m_2/d \rfloor - (l_1+l_2)/d}.
 \end{split}
 \]
 (b) For the same reason as before, we may assume that $m_2 \leq n$.
  Noting that the correspondence
 $X\mapsto \frac{1}{d} X$ defined above preserves the cardinality and using an argument similar to
 the one in part (a), we obtain the following identity
 \[
  \binom{m_1+m_2-l_1-l_2}{k-2}=
  \sum_{d|(l_1,l_2,n)}\Psi_k (l_1/d,\lfloor m_1/d \rfloor,l_2 /d,\lfloor m_2/d \rfloor, n/d )
  \]
  which by Theorem \ref{thm:inversion} is equivalent to
  \[
   \Psi_k (l_1,m_1,l_2,m_2,n) =
   \sum_{d|(l_1,l_2,n)}\mu(d)\binom{\lfloor m_1/d \rfloor + \lfloor m_2/d \rfloor - (l_1+l_2)/d }{k-2}.
 \]
\end{proof}
 By arguments similar to the ones in the proof of Lemma \ref{lemma1} we have:
 \begin{lemma} \label{lemma2}
 Let
 \[
 \psi(l,m,n) = \# \{X\subseteq [l,m]:\ l \in X\ \text{and\ } \gcd(X,n)=1 \},
 \]
 \[
 \psi_k(l,m,n) = \# \{X\subseteq [l,m]:\ l \in X,\ \# X=k,\ \text{and\ } \gcd(X,n)=1 \}.
 \]
 Then
 \[
 \psi(l,m,n) = \sum_{d|(l,n)} \mu(d) 2^{\lfloor m/d \rfloor - l/d} ,
 \]
 \[
 \psi_k(l,m,n) = \sum_{d|(l,n)} \mu(d) \binom{\lfloor m/d \rfloor - l/d}{k}.
 \]
 \end{lemma}
 We are now ready to prove the main theorem of this section.
 \begin{theorem} \label{main2}
 We have
 \[ (a)\quad
 \Phi([l_1,m_1]\cup[l_2,m_2]) = \sum_{d|n}\mu(d)
  2^{\lfloor \frac{m_1}{d} \rfloor + \lfloor \frac{m_2}{d} \rfloor -
    \lfloor \frac{l_1-1}{d} \rfloor - \lfloor \frac{l_2-1}{d} \rfloor},
 \]
 \[ (b)\quad
 \Phi_k([l_1,m_1]\cup[l_2,m_2]) = \sum_{d|n}\mu(d)
  \binom{\lfloor \frac{m_1}{d} \rfloor + \lfloor \frac{m_2}{d} \rfloor -
    \lfloor \frac{l_1-1}{d} \rfloor - \lfloor \frac{l_2-1}{d} \rfloor}{k}.
 \]
 \end{theorem}
 \begin{proof}
 (a)\ Clearly
 \begin{multline*}
 \Phi([l_1,m_1]\cup[l_2,m_2]) = \\ (1)\quad\quad\quad\qquad\qquad
 \Phi([1,m_1] \cup [l_2,m_2]) -
 \sum_{i=1}^{l_1 -1} \sum_{j=l_2}^{m_2} \Psi(i,m_1,j,m_2,n) -
 \sum_{i=1}^{l_1-1} \psi(i,m_1,n) = \\
 \sum_{d|n}\mu(d) 2^{\lfloor \frac{m_1}{d} \rfloor + \lfloor \frac{m_2}{d} \rfloor -
    \lfloor \frac{l_2-1}{d} \rfloor} -
 \sum_{i=1}^{l_1 -1} \sum_{j=l_2}^{m_2} \sum_{d|(i,j,n)} \mu(d)
    2^{\lfloor \frac{m_1}{d} \rfloor + \lfloor \frac{m_2}{d} \rfloor - \frac{i+j}{d}} -
 \sum_{i=1}^{l_1-1} \sum_{d|(i,n)} \mu(d) 2^{\lfloor \frac{m_1}{d} \rfloor - \frac{i}{d}},
 \end{multline*}
 where the second identity follows by Theorem \ref{thm:main1}, Lemma \ref{lemma1}, and
 Lemma \ref{lemma2}.
 Rearranging the triple summation in identity (1), we get
 \[
 \begin{split}
 \sum_{i=1}^{l_1 -1} \sum_{j=l_2}^{m_2} \sum_{d|(i,j,n)} \mu(d)
    2^{\lfloor \frac{m_1}{d} \rfloor + \lfloor \frac{m_2}{d} \rfloor - \frac{i+j}{d}}
 &=
  \sum_{d|n}\sum_{\substack{i=1 \\ d|i}}^{l_1 -1} \sum_{\substack{j=l_2\\d|j}}^{m_2}
  \mu(d) 2^{\lfloor \frac{m_1}{d} \rfloor + \lfloor \frac{m_2}{d} \rfloor-\frac{i+j}{d}} \\
  &=
 \sum_{d|n} \mu(d) 2^{\lfloor \frac{m_1}{d} \rfloor + \lfloor \frac{m_2}{d} \rfloor}
  \sum_{i=1}^{\lfloor\frac{l_1-1}{d}\rfloor} 2^{-i}
   \sum_{j=\lfloor \frac{l_2-1}{d} \rfloor +1}^{\lfloor \frac{m_2}{d} \rfloor} 2^{-j} \\
 (2) \qquad\qquad\qquad \quad\qquad\qquad\qquad
 &=
  \sum_{d|n} \mu(d) 2^{\lfloor \frac{m_1}{d} \rfloor + \lfloor \frac{m_2}{d} \rfloor - \lfloor\frac{l_2-1}{d}\rfloor}
   (1- 2^{-\lfloor \frac{m_2}{d} \rfloor + \lfloor\frac{l_2 -1}{d} \rfloor})
   \sum_{i=1}^{\lfloor\frac{l_1-1}{d}\rfloor} 2^{-i} \\
 & =
  \sum_{d|n} \mu(d) 2^{\lfloor \frac{m_1}{d} \rfloor + \lfloor \frac{m_2}{d} \rfloor - \lfloor\frac{l_2-1}{d}\rfloor}
  (1- 2^{-\lfloor \frac{m_2}{d} \rfloor + \lfloor\frac{l_2 -1}{d} \rfloor})
  (1- 2^{-\lfloor \frac{l_1-1}{d} \rfloor}).
 \end{split}
 \]
 Similarly the double summation in identity (1) gives
 \[ (3)\qquad\qquad\qquad\qquad
 \sum_{i=1}^{l_1-1} \sum_{d|(i,n)} \mu(d) 2^{\lfloor \frac{m_1}{d} \rfloor - \frac{i}{d}}
 = \sum_{d|n} \mu(d) 2^{\lfloor \frac{m_1}{d} \rfloor}(1- 2^{-\lfloor \frac{l_1-1}{d} \rfloor}).
 \]
 Combining identities (1), (2), and (3) we find
 \[
 \Phi([l_1,m_1]\cup[l_2,m_2]) = \sum_{d|n}\mu(d)
 2^{\lfloor \frac{m_1}{d} \rfloor + \lfloor \frac{m_2}{d} \rfloor -
    \lfloor \frac{l_1-1}{d} \rfloor - \lfloor \frac{l_2-1}{d} \rfloor}.
 \]
 This completes the proof of part (a). Part (b) follows by similar ideas.
 \end{proof}
 \begin{corollary}\label{corollary1} \emph{(\cite{Ayad-Kihel2})}
 We have
 \[ \Phi([l,m],n)= \sum_{d|n}2^{\lfloor\frac{m}{d}\rfloor - \lfloor\frac{l-1}{d}\rfloor}, \]
 \[ \Phi_k([l,m],n)= \sum_{d|n}\binom{\lfloor\frac{m}{d}\rfloor - \lfloor\frac{l-1}{d}\rfloor}{k}. \]
 \end{corollary}
 \begin{proof}
 Use Theorem \ref{main2} with $l_1=l$, $m_1=m-1$, and $l_2=m_2=m$.
 \end{proof}

 \section{Combinatorial identities}
 In this section we assume that $n >1$.
 \begin{theorem} \label{identities1}
 We have:
 \[ (a)\quad
  \sum_{d|n} \mu(d) 2^{\lfloor \frac{m}{d} \rfloor - \lfloor \frac{m-1}{d} \rfloor} =
  \sum_{d|n} \mu(d) \left( \left\lfloor \frac{m}{d} \right\rfloor -
  \left\lfloor \frac{m-1}{d} \right\rfloor \right) =
  \begin{cases}
  0,\ \text{if $\gcd(m,n) >1$} \\
  1,\ \text{if $\gcd(m,n) =1$}.
  \end{cases}
  \]
  \begin{multline*} (b)\quad
  \sum_{d|n} \mu(d) 2^{\lfloor \frac{m+1}{d} \rfloor - \lfloor \frac{m-1}{d} \rfloor} = \\
  \begin{cases}
  1,\ \text{if $(m,n)>1$ and $(m+1,n)>1$}\\
  2,\ \text{if $(m,n)=1$ and $(m+1,n)>1$ or $(m,n)>1$ and
           $(m+1,n)=1$} \\
  3,\ \text{if $(m,n)=(m+1,n)=1$}.
  \end{cases}
  \end{multline*}
 \end{theorem}
 \begin{proof}
(a) Apply Corollary~\ref{corollary1} to $l=m$ and $k=1$. \\
(b) Apply Corollary~\ref{corollary1} to the interval $[m,m+1]$.
 \end{proof}
 \begin{theorem} \label{identities2}
 We have:
 \begin{multline*} (a)\quad
 \sum_{d|n} \mu(d) \left(\left\lfloor \frac{m+1}{d} \right\rfloor - \left\lfloor \frac{m-1}{d} \right\rfloor \right)= \\
 \begin{cases}
  0,\ \text{if $(m,n)>1$ and $(m+1,n)>1$}\\
  1,\ \text{if $(m,n)=1$ and $(m+1,n)>1$ or $(m,n)>1$ and
           $(m+1,n)=1$} \\
  2,\ \text{if $(m,n)=(m+1,n)=1$}.
  \end{cases}
  \end{multline*}
  \begin{multline*} (b)\quad
  \sum_{d|n} \mu(d) \binom{\lfloor \frac{m+1}{d} \rfloor - \lfloor \frac{m-1}{d} \rfloor +1}{2} = \\
  \begin{cases}
  1,\ \text{if $(m,n)>1$ and $(m+1,n)>1$}\\
  2,\ \text{if $(m,n)=1$ and $(m+1,n)>1$ or $(m,n)>1$ and
           $(m+1,n)=1$} \\
  3,\ \text{if $(m,n)=(m+1,n)=1$}.
  \end{cases}
  \end{multline*}
 \end{theorem}
 \begin{proof}
 (a) Apply Theorem \ref{main2}(b) to $l_1=m$, $m_1=m+1$, $l_2=m_2=n$, and $k=1$ and use the fact that
 $\sum_{d|n} \mu(d) = 0$ whenever $n>1$. \\
 (b) Apply Theorem \ref{main2}(b) to $l_1=m$, $m_1=m+1$, $l_2=m_2=n$, and $k=2$.
 \end{proof}
 %


\end{document}